\documentclass[11pt]{amsart}
\usepackage{amsfonts}
\usepackage{amsmath}
\usepackage{amssymb}
\usepackage{graphicx}

\newcommand{\ent}{{\mathbb{Z}}}

\newcommand{\real}{{\mathbb{R}}}
\newcommand{\comp}{{\mathbb{C}}}
\newcommand{\quat}{{\mathbb{H}}}
\newcommand{\oct}{{\mathbb{O}}}
\newcommand{\dist}{{\mathcal{H}}}

\newcommand{\s}{\vspace{0.5cm}}

\newtheorem{theo}{Theorem}
\newtheorem{lemma}{Lemma}
\newtheorem{defi}{Definition}

\newcommand{\spn}{{\rm{span}}}

\title{Sub-Riemannian geometry of parallelizable spheres}
\author[Mauricio Godoy M.,Irina Markina]{Mauricio Godoy Molina\\ Irina Markina}

\address{Department of Mathematics, University of Bergen, Norway.}
\email{mauricio.godoy@math.uib.no}

\address{Department of Mathematics, University of Bergen, Norway.}
\email{irina.markina@uib.no}

\thanks{The authors are partially supported by the grant of the
Norwegian Research Council \# 177355/V30 and by the grant of the
European Science Foundation Networking Programme HCAA}

\subjclass[2000]{53C17, 32V15}

\keywords{sub-Riemannian geometry, CR geometry, Hopf bundle,
Ehresmann connection, parallelizable spheres, quaternions,
octonions}

\begin{document}

\maketitle

\begin{abstract}

The first aim of the present paper is to compare various sub-Riemannian
structures over the three dimensional sphere $S^3$ originating
from different constructions. Namely, we describe the
sub-Riemannian geometry of $S^3$ arising through its right Lie
group action over itself, the one inherited from the natural
complex structure of the open unit ball in $\comp^2$ and the
geometry that appears when considering the Hopf map as a principal
bundle. The main result of this comparison is that in fact those
three structures coincide.

In the second place, we present two bracket generating
distributions for the seven dimensional sphere $S^7$ of step 2
with ranks 6 and 4. These yield to sub-Riemannian structures for
$S^7$ that are not present in the literature until now. One of the
distributions can be obtained by considering the CR geometry of
$S^7$ inherited from the natural complex structure of the open
unit ball in $\comp^4$. The other one originates from the
quaternionic analogous of the Hopf map.

\end{abstract}

\section{Introduction}

One of the main objectives of classical sub-Riemannian geometry is
to study manifolds which are path-connected by curves which are
admissible in a certain sense. In order to define what does
admissibility mean in this context, we begin by setting
notations that will be used throughout this paper. Let $M$ be a
smooth connected manifold of dimension $n$, together with a smooth
distribution $\dist\subset TM$ of rank $2\leq k<n$. Such vector bundles
are often called {\em{horizontal}} in the literature. An
absolutely continuous curve $\gamma:[0,1]\to M$ is called
{\em{admissible}} or {\em{horizontal}} if $\dot\gamma(t)\in\dist$
a.e.

Distributions satisfying the condition that their Lie-hull equals
the whole tangent space of the manifold at each point play a
central role in the search for horizontally path-connected
manifolds. Such vector bundles are said to satisfy the {\it
bracket generating condition}. To be more precise, define the
following real vector bundles
$$\dist^1=\dist,\quad\quad\dist^{r+1}=[\dist^r,\dist]+
\dist^r\quad\mbox{for }r\geq1,$$ which naturally give rise to the
flag $$\dist=\dist^1\subseteq\dist^2\subseteq\dist^3\subseteq
\ldots.$$ Then we say that a distribution is bracket generating if
for all $x\in M$ there is an $r(x)\in\ent^+$ such that
\begin{equation}\label{brgen}\dist_x^{r(x)}=T_xM.\end{equation}
If the dimensions $\dim\dist^r_x$ do not depend on $x$ for any $r\geq 1$, we say that
$\dist$ is a regular distribution. The least $r$
such that~\eqref{brgen} is satisfied is called the step
of $\dist$. We will focus on
regular distributions of step 2. In~\cite{LiuSussmann} the reader
can find detailed definitions and broad discussion about terminology.

The following classical result shows the precise relation between
the notion of path-connectedness by means of horizontal curves and
the assumption that $\dist$ is a bracket generating distribution.

\begin{theo}[\cite{Chow, R}]\label{Chow}
If a distribution $\dist\subset TM$ is bracket generating, then
the set of points that can be connected to $p\in M$ by a
horizontal path is the connected component of $M$ containing $p$.
\end{theo}

Thus, the search for horizontally path-connected manifolds can
be reduced to the search of appropriate distributions defined on
them. We define the class of manifolds which will be of our concern.

\begin{defi}
A sub-Riemannian structure over a manifold $M$ is a pair $(\dist,
\langle\cdot,\cdot\rangle)$, where $\dist$ is a bracket generating distribution
and $\langle\cdot,\cdot\rangle$ a fiber inner product defined on
$\dist$. In this setting, the length of an absolutely continuous
horizontal curve $\gamma:[0,1]\to M$ is $$\ell(\gamma):= \int_0^1
\|\dot\gamma(t)\|dt,$$ where $\|\dot\gamma(t)\|^2= \langle\dot
\gamma(t),\dot\gamma(t)\rangle$ whenever $\dot\gamma(t)$ exists.
The triple $(M,\dist,\langle\cdot,\cdot\rangle)$ is called
sub-Riemannian manifold.
\end{defi}

Thereby, restricting ourselves to connected sub-Riemannian
manifolds endowed with bracket generating distributions, it is
possible to define the notion of sub-Riemannian distance between
two points.

\begin{defi}
The sub-Riemannian distance $d(p,q)\in\real_{\geq0}$ between two
points $p,q\in M$ is given by $d(p,q):=\inf\ell(\gamma)$, where
the infimum is taken over all absolutely continuous horizontal
curves joining $p$ to $q$.

An absolutely continuous horizontal curve that realizes the
distance between two points is called a horizontal length
minimizer.
\end{defi}

\vspace{0.3cm}

\paragraph{\bf{Remark:}} The connectedness of $M$ and the fact
that $\dist$ is bracket generating, assure that $d(p,q)$ is a
finite nonnegative number. Nevertheless, the bracket generating
hypothesis, required for the previous definition, is possible to be
weakened. In fact, in \cite{Suss} the author finds a necessary and
sufficient requirement to horizontal path-connectedness for a
manifold in terms of the corresponding distribution. Clearly, this
theorem contains, as a particular case, the bracket generating
condition.

\s

Historically, the first examples of sub-Riemannian manifolds that
have been considered were Lie groups, see
e.g.~\cite{ABGR,BR,CM,GV,HR}, because due to its algebraic
structure, it is possible to reduce the problem of finding
globally defined bracket generating distributions to finding such
distributions at the identity by considering right (or left)
invariant vector fields. An important role has been played by
considering domains in Euclidean spaces with special algebraic
structures (such as the Heisenberg groups, $\quat-$type groups as
their natural generalizations to Clifford algebras, Engel groups,
Carnot groups, etc.). Particular attention have had the three
dimensional unimodular Lie groups which were studied, for example,
in \cite{ABGR}, \cite{BR} and \cite{GV}. The main purpose of this
communication is to present recent results concerning different
sub-Riemannian structures of the second simplest family of
examples of manifolds, namely, spheres. We also be inspired by the
article~\cite{Urb}, where the close relation between the Hopf map
and physical applications is presented.

The following celebrated theorem in topology, see \cite{Adams},
gives a very strong restriction on the problem of finding globally
defined sub-Riemannian structures over spheres.

\begin{theo}[\cite{Adams}]\label{Adams}
Let $S^{n-1}=\{x\in\real^n:\|x\|^2=1\}$ be the unit sphere in
$\real^n$, with respect to the usual Euclidean norm $\|\cdot\|$.
Then $S^{n-1}$ has precisely $\varrho(n)-1$ globally defined and
non vanishing vector fields, where $\varrho(n)$ is defined in the
following way: if $n=(2a+1)2^b$ and $b=c+4d$ where $0\leq
c\leq3$, then $\varrho(n)=2^c+8d$.

In particular, two classical consequences follow: $S^1$, $S^3$ and
$S^7$ are the only spheres with maximal number of linearly
independent globally defined non vanishing vector fields, and the
even dimensional spheres have no globally defined and non vanishing
vector fields.
\end{theo}

The condition that a manifold $M$ has maximal number of linearly
independent globally defined non vanishing vector fields is usually
rephrased as saying that $M$ is {\em parallelizable}. The fact that
$S^1$, $S^3$ and $S^7$ are the only parallelizable spheres was
proven in \cite{BM} and that the even dimensional spheres have no
globally defined and non vanishing vector fields is a consequence of
the Hopf index theorem, see \cite{Vick}.

This theorem permits to conclude at least two major points of
discussion: there is no possible global basis of a sub-Riemannian
structure for spheres with even dimension and it is impossible to
find a globally defined basis for bracket generating
distributions, except for $S^3$ and $S^7$. The fact that $S^3$ and
$S^7$ can be seen as the set of quaternions and octonions of unit
length will play a core role in many arguments throughout this
paper.

The main results that we present here are: a comparison between
three sub-Riemannian structures of $S^3$, namely, the one arising
through its right Lie group action over itself as the set of unit
quaternions, the one inherited from the natural complex structure
of the open unit ball in $\comp^2$ and the geometry that appears
when considering the Hopf map as a principal bundle. Notice that
this structure admits a tangent cone isomorphic to the one
dimensional Heisenberg group in a sense of
Gromov-Margulis-Mitchell-Mostow construction of the tangent
cone~\cite{Grom,Mitch,M}. Considering CR-structure of $S^7$,
inherited from the natural complex structure of the open unit ball
in $\comp^4$, we obtain a 2-step bracket generating distribution
of the rank 6. This construction intimately related to the Hopf
fibration $S^1\to S^7\to \mathbb CP^3$. Making use of the
quaternionic analogue of the Hopf map $S^3\to S^7\to S^4$, we
present another 2-step bracket generating distribution that has
the rank~4. We conclude that the sphere $S^7$ admits two
principally different sub-Riemannian structures. The tangent cone,
in the first case, is isomorphic to the 3-dimensional Heisenberg
group, and in the second case it has a structure of the
quaternionic Heisenberg-type group with 3-dimensional
center~\cite{CM}. In both cases we present the basis of the
horizontal distribution that is very useful in future studies of
geodesics and hypoelliptic operators related to the spherical
sub-Riemannian manifolds.

\section{$S^3$ as a sub-Riemannian manifold}\label{S3sR}

Throughout this paper $\quat$ will denote the quaternions, that
is, $\quat=(\real^4,+,\circ)$ where $+$ stands for the usual
coordinate-wise addition in $\real^4$ and $\circ$ is a non-commutative
product given by the formula

$$(x_0+x_1i+x_2j+x_3k)\circ(y_0+y_1i+y_2j+y_3k)=$$
$$=(x_0y_0-x_1y_1-x_2y_2-x_3y_3)+(x_1y_0+x_0y_1-x_3y_2+x_2y_3)i+$$
$$+(x_2y_0+x_3y_1+x_0y_2-x_1y_3)j+(x_3y_0-x_2y_1+x_1y_2+x_0y_3)k.$$

It is important to recall that $\quat$ is a non-commutative,
associative and normed real division algebra. Let $q=t+ai+bj+ck
\in\quat$, then the conjugate of $q$, is given by $$\bar q=
t-ai-bj-ck.$$ We define the norm $|q|$ of $q\in\quat$, as for
the complex numbers, by $|q|^2=q\bar q$.

Given the fact that the sphere $S^3$ can be realized as
the set of unit quaternions, it has a non abelian Lie
group structure induced by quaternion multiplication.

The multiplication rule in $\quat$ induces a
right translation $R_y(x)$ of an element $x=x_0+x_1i+x_2j+x_3k$ by
the element $y=y_0+y_1i+y_2j+y_3k$. The right invariant basis
vector fields are defined as $Y(y)=Y(0)(R_y(x))_*$, where $Y(0)$
are the basis vectors at the unity of the group. The matrix
corresponding to the tangent map $(R_y(x))_*$, obtained by the
multiplication rule, becomes

$$(R_y(x))_*=\left(\begin{array}{cccc}y_0&y_1&y_2&y_3\\
-y_1&y_0&-y_3&y_2\\-y_2&y_3&y_0&-y_1\\-y_3&-y_2&y_1&y_0\end{array}\right).$$

Calculating the action of $(R_y(x))_*$ in the basis of unit
vectors of $\real^4$ we get the four vector fields
\begin{eqnarray*}
N(y)&=&y_0\partial_{y_0}+y_1\partial_{y_1}+y_2\partial_{y_2}+y_3\partial_{y_3}\\
V(y)&=&-y_1\partial_{y_0}+y_0\partial_{y_1}-y_3\partial_{y_2}+y_2\partial_{y_3}\\
X(y)&=&-y_2\partial_{y_0}+y_3\partial_{y_1}+y_0\partial_{y_2}-y_1\partial_{y_3}\\
Y(y)&=&-y_3\partial_{y_0}-y_2\partial_{y_1}+y_1\partial_{y_2}+y_0\partial_{y_3}.
\end{eqnarray*}

It is easy to see that $N(y)$ is the unit normal to $S^3$ at $y\in
S^3$ with respect to the usual Riemannian structure $\langle\cdot,
\cdot\rangle$ in $T\,\real^4$. Moreover, for any $y\in S^3$ $$\langle
N(y),V(y)\rangle_y= \langle N(y),X(y)\rangle_y=\langle N(y),Y(y)
\rangle_y=0$$ and $$\langle N(y),N(y)\rangle_y=\langle V(y),V(y)
\rangle_y=\langle X(y),X(y)\rangle_y=\langle Y(y),Y(y)\rangle_y=1.$$

Since the matrix $$\left(\begin{array}{cccc}-y_1&y_0&-y_3&y_2\\
-y_2&y_3&y_0&-y_1\\-y_3&-y_2&y_1&y_0\end{array}\right)$$ has rank
three, we conclude that the vector fields $\{V(y),X(y),Y(y)\}$ form
an orthonormal basis of $T_yS^3$ with respect to $\langle\cdot,
\cdot\rangle_y$.

Observing that $[X,Y]=2V$, we see that the distribution $\spn
\{X,Y\}$ is bracket generating, therefore it satisfies the
hypothesis of Theorem~\ref{Chow}. The geodesics of the left
invariant sub-Riemannian structure of $S^3$ are determined
in~\cite{CMV}, while in \cite{HR} the same results are achieved by
considering the right invariant structure of $S^3$.

Notice that the distribution $\spn \{X,Y\}$ can also be defined as
the kernel of the contact one form $$\omega=-y_1\,dy_0+y_0\,
dy_1-y_3\,dy_2+y_2\,dy_3.$$

\s

\paragraph{\bf{Remark:}} It is easy to see that $[V,Y]=2X$ and
$[X,V]=2Y$, therefore the distributions $\spn\{Y,V\}$ and
$\spn\{X,V\}$ are also bracket generating. The corresponding contact forms are $$\theta=
-y_2\,dy_0+y_3\,dy_1+y_0\,dy_2-y_1\,dy_3$$ and $$\eta=
-y_3\,dy_0-y_2\,dy_1+y_1\,dy_2+y_0\,dy_3$$ respectively.
This means that there
is a priori no natural choice of a sub-Riemannian structure on
$S^3$ generated by the Lie group action of multiplication of quaternions. Any choice that can be made, will produce essentially
the same geometry.

\section{$S^3$ as a CR manifold}\label{S3CR}

Consider $S^3$ as the boundary of the unit ball $B^4$ on
$\comp^2$, that is, as the hypersurface $$S^3:=\{(z,w)\in\comp^2
:z\bar z+w\bar w= 1\}.$$ The sphere $S^3$ cannot be endowed with
a complex structure, but nevertheless it possess a differentiable
structure compatible with the natural complex structure of the ball $B^4$
as an open set in $\comp^2$. We will show that this differentiable
structure over the sphere $S^3$ (CR structure) is equivalent to the sub-Riemannian
one considered in the previous section.
We begin by recalling the definition of a CR structure, according to
\cite{BOG}.

\begin{defi}
Let $W$ be a real vector space. A linear map $J:W\to W$ is called an almost
complex structure map if $J\circ J=-I$, where $I:W\to W$ is the
identity map.
\end{defi}

In the case $W=T_p\real^{2n}$, $p=(x_1,y_1,\ldots,x_n,y_n)
\in\real^{2n}$, we say that the standard almost complex structure for $W$
is defined by setting $$J_n(\partial_{x_j})=\partial_{y_j},\quad\quad
J_n(\partial_{y_j})=-\partial_{x_j},\quad\quad 1\leq j\leq n.$$

For a smooth real submanifold $M$ of $\comp^n$ and a point $p\in M$,
in general the tangent space $T_pM$ is not invariant under the almost
complex structure map $J_n$ for $T_p(\comp^n)$. We are interested in
the largest subspace invariant under the action of $J_n$.

\begin{defi}
For a point $p\in M$, the complex or holomorphic tangent space of
$M$ at $p$ is the vector space $$H_pM=T_pM\cap J_n(T_pM).$$
\end{defi}

In this setting, the following result takes place (see \cite{BOG}).

\begin{lemma}\label{lemmaCR}
Let $M$ be a real submanifold of $\comp^n$ of real dimension $2n-d$.
Then $$2n-2d\leq\dim_\real H_pM\leq 2n-d,$$ and $\dim_\real H_pM$
is an even number.
\end{lemma}

A real submanifold $M$ of $\comp^n$ is said to have a CR structure
if $\dim_\real H_pM$ does not depend on $p\in M$. In particular,
by the previous lemma, every smooth real hypersurface
$S$ embedded in $\comp^n$ satisfies $\dim_\real M_pS=2n-2$,
therefore $S$ is a CR manifold. This fact is applied to every odd
dimensional sphere.

The question addressed now is to describe $H_pS^3$. By the
discussion in the previous paragraph, $H_pS^3$ is a complex vector
space of complex dimension one. This description can be achieved
by considering the differential form $$\omega=\bar z dz+\bar w
dw$$ and observing that $\ker\omega$ is precisely the set we are
looking for. Straightforward calculations show that
$$\ker\omega=\spn\{\bar w\partial_z-\bar z\partial_w\}.$$

In real coordinates this corresponds to $$\bar w\partial_z-\bar z
\partial_w =\frac12(-X+iY),$$ where $X$ and $Y$ were defined in
Section \ref{S3sR}. It is important to remark that this is
precisely the maximal invariant $J_2-$subspace of $T_pS^3$, namely
$$J_2(X)=Y, \quad\quad J_2(Y)=-X,$$ then
$J_2(\spn\{X,Y\})=\spn\{X,Y\}$, but $J_2(V)=-N\notin T_pS^3$ for
any point $p\in S^3$. Therefore, the distribution corresponding to
the right invariant action of $S^3$ over itself is the same to its
one dimensional (complex) CR structure.

\s

\paragraph{\bf{Remark:}} The distribution associated to the
anti-holomorphic form $$\bar\omega=zd\bar z+wd\bar
w$$ is the same as the previous one. More explicitly: $$\ker
\omega=\spn\{-w\partial_{\bar z}+z\partial_{\bar w}\}$$ and in
real coordinates this corresponds to $$-w\partial_{\bar z}+z
\partial_{\bar w}=\frac12(X+iY).$$

\s

The same almost complex structure as the previously described can
be obtained by means of the covariant derivative of $S^3$
considered as a smooth Riemannian manifold embedded in $\mathbb
R^4$. Namely, in \cite{HR} it is introduced the mapping
$J(X)=D_XV$, were $D$ denotes the Levi-Civit\'a connection on
tangent bundle to $S^3$ and $V$ is the vector field defined in
Section~\ref{S3sR}.

\section{$S^3$ as principal bundle}\label{S3pb}

In this section we describe how the structure of a principal
$S^1-$bundle over $S^3$ induces a bracket generating distribution
on $S^3$. Namely, it is possible to consider $S^3$ as a $S^1-$space,
according to the action $$\lambda\cdot(z,w)=(\lambda z,\lambda w),$$ where $\lambda\in S^1=\{v\in\comp:|v|^2=1\}$ and $(z,w)\in
S^3=\{(z,w)\in\comp^2:|z|^2+|w|^2=1\}$.

Consider the Hopf map $h:S^3\to S^2$ as a principal
$S^1-$bundle, see~\cite{H}, given explicitly by $$h(z,w)=(|z|^2-|w|^2,2z\bar
w),$$ where $S^2=\{(x,\zeta)\in\real\times\comp:x^2+|\zeta|^2=
1\}$. Clearly, $h$ is a submersion of $S^3$ onto $S^2$, and it is
a bijection between $S^3/S^1$ and $S^2$, where $S^3/S^1$ is
understood as the orbit space of the $S^1-$action over $S^3$,
previously defined.

Let $p=(x_0,\zeta_0)\in S^2$. It is easy to verify that
$h^{-1}(p)= (z_0,w_0)$ mod $S^1$, where $(z_0,w_0)$ is one
preimage of $p$ under $h$. Consider the great
circle $$\gamma_p(t)=e^{2\pi it} (z_0,w_0),\quad\quad
t\in[0,1],$$ in $S^3$, that projects to $p$ under the Hopf map.
Now consider the following vector field associated to
$\gamma_p$, defined by the tangent vectors $$\dot
\gamma_p(t)=2\pi ie^{2\pi it}(z_0,w_0)\in
T_{\gamma_p(t)}S^3.$$

We write the Hopf map in real coordinates, where, in particular,
$\gamma_p(t)=(z(t),w(t))=(x_0(t)+ix_1(t),x_2(t)+ix_3(t))=
(x_0(t),x_1(t),x_2(t),x_3(t))$. It is easy to see that
$$[d_{\gamma_p(t)}h]=2\left(\begin{array}{cccc}
x_0(t)&x_1(t)&-x_2(t)&-x_3(t)\\x_2(t)&x_3(t)&x_0(t)&x_1(t)\\
-x_3(t)&x_2(t)&x_1(t)&-x_0(t)\end{array}\right).$$ Thus, the
Hopf map induces the following action over the vector field
$\dot\gamma_p(t)$: $$[d_{\gamma_p(t)}h]\dot\gamma_p(t)
=4\pi\left(\!\!\begin{array}{rrrr}x_0(t)&x_1(t)&-x_2(t)&-x_3(t)\\
x_2(t)&x_3(t)&x_0(t)&x_1(t)\\-x_3(t)&x_2(t)&x_1(t)&-x_0(t)
\end{array}\!\!\right)\!\!\!\left(\!\!\begin{array}{c}\dot x_0(t)\\
\dot x_1(t)\\\dot x_2(t)\\\dot x_3(t)\end{array}\!\!\right)=
\left(\!\!\begin{array}{c}0\\0\\0\\0\end{array}\!\!\right).$$
Therefore, if $[d_{\gamma_p(t)}h]$ is a full rank matrix, we would have
characterized the kernel of it, by \begin{equation}\label{kerhopf1}\ker d_{\gamma_p(t)}h=
\spn\{\dot\gamma_p(t)\}.\end{equation} Notice that, using the notation of
Section \ref{S3sR}, the following identity holds
\begin{eqnarray}\label{gammaV}
\dot\gamma_p(t)=2\pi V(\gamma_p(t)).
\end{eqnarray}
There is a simple way to see that the
matrix $[d_{\gamma_p(t)}h]$ is full rank. Denote by $D_i$ the
$3\times 3$ matrix obtained by deleting the $i-$th column of
$[d_{\gamma_p(t)}h]$, $i=1,2,3$ or $4$, then $$\det(D_1)^2+
\det(D_2)^2+\det(D_3)^2+\det(D_4)^2=$$ $$=\left((x_0(t))^2+
(x_1(t))^2+(x_2(t))^2+(x_3(t))^2\right)^3=1$$ implies that
$[d_{\gamma_p(t)}h]$ is full rank.

Before describing how the Hopf map induces a horizontal
distribution, it is necessary to present some definitions.

\begin{defi}[Ehresmann Connection~\cite{M}, chapter 11]\label{EC}
Let $M$ and $Q$ be two differentiable manifolds, and let $\pi:Q\to
M$ be a submersion. Denoting by $Q_m=\pi^{-1}(m)$ the fiber
through $m\in M$, the {\em vertical space} at $q$ is the tangent
space at the fiber $Q_{\pi(q)}$ and it is denoted by $V_q$.

An {\em Ehresmann connection} for the submersion $\pi:Q\to M$ is a
distribution $\dist\subset TQ$ which is everywhere transversal to
the vertical, that is: $$V_q\oplus\dist_q=T_qQ.$$
\end{defi}

We apply Definition~\ref{EC} to $S^3$ in order to define the
Ehresmann connection. Since we know that $\ker d_p
h=\spn\{V(p)\}$, for every $p\in S^3$ by~\eqref{kerhopf1}
and~\eqref{gammaV}, and moreover, $$\langle
X(p),V(p)\rangle_p=\langle Y(p),V(p)\rangle_p=\langle
X(p),Y(p)\rangle_p=0$$ where $\langle\cdot,\cdot\rangle_p$ stands
for the usual Riemannian structure defined at $p\in S^3$, we see
that
\begin{eqnarray}\label{HS3}
\dist_p=\spn\{X(p),Y(p)\}
\end{eqnarray}
is an Ehresmann connection for the submersion $h:S^3\to S^2$ with
$V(p)$ as a vertical space.

\begin{defi}
 Let $G$ be a Lie group acting on $Q$ and $\pi:Q\to M$ a
submersion, with Ehresmann connection ${\mathcal{H}}$, which is a
fiber bundle with fiber~$G$. The submersion $\pi$ is called a
principal $G-$bundle with connection, if the following
conditions hold:

\begin{itemize}
\item $G$ acts freely and transitively,

\item the group orbits are the fibers of $\pi:Q\to M$ (thus $M$ is
isomorphic to $Q/G$ and $\pi$ is the canonical projection) and

\item the $G-$action on $Q$ preserves the horizontal distribution
${\mathcal{H}}$;
\end{itemize}

\end{defi}

We conclude that the Hopf fibration is a principal $S^1-$bundle
with connection $\dist$, defined in~\eqref{HS3}.

\begin{defi}
A sub-Riemannian metric $(\dist,\langle\cdot,\cdot\rangle)$ on
the principal $G$-bundle $\pi:G\to M$ is called a metric of bundle
type if the inner product $\langle\cdot,\cdot\rangle$ on the
horizontal distribution $\dist$ is induced from a Riemannian metric
on $M$.
\end{defi}

The sub-Riemannian metric $\langle\cdot,\cdot\rangle|_\dist$,
obtained by restricting the usual Riemannian metric of $S^3$
to the distribution $\dist$ is, by construction, a metric of
bundle type.

Thus, the Hopf map indicates in a very natural topological way how
to make a natural choice of the horizontal distribution $\dist$ that was not obvious when we considered the right action of $S^3$ over itself.

\paragraph{\bf{Remark:}} Observe that the considered vector fields
coincide with the right invariant vector fields. This phenomenon does not appear when
we change the right action to the left action of $S^3$ over itself.

\section{Tangent vector fields for $S^7$}

In Sections~5 - 7 we study different sub-Riemannian structures
over the sphere $S^7$, making use of ideas of Sections 2 - 4. As a
result, we obtain two principally different types of horizontal
distributions. One of them of rank $6$ and other of the rank $4$.
Moreover, as we shall see, the sub-Riemannian structure induced by
the CR-structure and quaternionic analogue of the Hopf map are
essentially different. We start from the construction of the
tangent vector fields to $S^7$.

The multiplication of unit octonions is not associative, therefore $S^7$ is not a group in a contrast with $S^3$. Nevertheless, we still able to use the multiplication law in order to fugue out global tangent vector fields. To do this, we present a multiplication table for the basis vectors of
$\real^8$. The non-associative multiplication gives rise to the
division algebra of the octonions $$\oct=\spn\{
e_0,e_1,e_2,e_3,e_4,e_5,e_6,e_7\}.$$

\begin{table}[h]
$$\begin{array}{c||c|c|c|c|c|c|c|c}
 &e_0&e_1&e_2&e_3&e_4&e_5&e_6&e_7\\\hline\hline
 e_0&e_0&e_1&e_2&e_3&e_4&e_5&e_6&e_7\\\hline
 e_1&e_1&-e_0&e_3&-e_2&e_5&-e_4&-e_7&e_6\\\hline
 e_2&e_2&-e_3&-e_0&e_1&e_6&e_7&-e_4&-e_5\\\hline
 e_3&e_3&e_2&-e_1&-e_0&e_7&-e_6&e_5&-e_4\\\hline
 e_4&e_4&-e_5&-e_6&-e_7&-e_0&e_1&e_2&e_3\\\hline
 e_5&e_5&e_4&-e_7&e_6&-e_1&-e_0&-e_3&e_2\\\hline
 e_6&e_6&e_7&e_4&-e_5&-e_2&e_3&-e_0&-e_1\\\hline
 e_7&e_7&-e_6&e_5&e_4&-e_3&-e_2&e_1&-e_0\\\hline
\end{array}
$$\caption{Multiplication table for the basis of $\oct$.}\label{oct}
\end{table}

According to the Table \ref{oct}, the formula for the product of
two octonions is presented in the Appendix, subsection \ref{multoct}.
This multiplication rule induces a matrix representation of the
right action of octonion multiplication, given explicitly by:

$$R_*=\left(\begin{array}{cccccccc}
y_0&-y_1&-y_2&-y_3&-y_4&-y_5&-y_6&-y_7\\
y_1&y_0&y_3&-y_2&y_5&-y_4&-y_7&y_6\\
y_2&-y_3&y_0&y_1&y_6&y_7&-y_4&-y_5\\
y_3&y_2&-y_1&y_0&y_7&-y_6&y_5&-y_4\\
y_4&-y_5&-y_6&-y_7&y_0&y_1&y_2&y_3\\
y_5&y_4&-y_7&y_6&-y_1&y_0&-y_3&y_2\\
y_6&y_7&y_4&-y_5&-y_2&y_3&y_0&-y_1\\
y_7&-y_6&y_5&y_4&-y_3&-y_2&y_1&y_0
\end{array}\right).
$$

We are able to find globally defined tangent vector fields which
are invariant under this action. We proceed by the analogy with
Section~2. The explicit formulae are given in
subsection~\ref{rivf} of the Appendix. The vector fields $\{Y_0,
\ldots,Y_7\}$ form a frame. More explicitly, we have that the
following identity holds $$\langle Y_i(y),Y_j(y)\rangle_y=
\delta_{ij},\quad\quad y\in S^7,\quad\quad i,j\in
\{0,1,\ldots,7\},$$ where $\langle\cdot,\cdot\rangle$ is the
standard Riemannian structure over $\real^8$.

\section{$CR$-structure and the Hopf map on $S^7$}

In the book~\cite{M} it is briefly discussed the general idea of
studying a sub-Riemannian geometry for odd dimensional spheres via
the higher Hopf fibrations. Namely, consider
$S^{2n+1}=\{z\in\comp^{n+1}:\|z\|^2=1\}$, then the $S^1-$action on
$S^{2n+1}$ given by
$$\lambda\cdot(z_0,\ldots,z_n)=(\lambda z_0,\ldots,\lambda z_n),$$
for $\lambda\in S^1$ and $(z_0,\ldots,z_n)\in S^{2n+1}$, induces
the well-known principal $S^1-$bundle $S^1\to S^{2n+1}\to\comp
P^n$ given explicitly by
$$S^{2n+1}\ni(z_0,\ldots,z_n)\mapsto[z_0:\ldots:z_n]\in\comp P^n$$
where $[z_0:\ldots:z_n]$ denote homogeneous coordinates. This map
is called higher Hopf fibration. The kernel of the map $h\colon
S^{2n+1}\to\comp P^n$ produces the vertical space and a
transversal to the vertical space distribution gives the Ehresmann
connection. We show that the vertical space is always given by an
action of standard almost complex structure on the normal vector
field to $S^{2n+1}$, and the Ehresmann connection coincides with
the holomorphic tangent space at each point of $S^{2n+1}$.

Theorem \ref{Adams} asserts that any odd dimensional sphere has at
least one globally defined non vanishing tangent vector field. If
the dimension of the sphere is of the form $4n+1$, then it has
only one globally defined non vanishing tangent vector field. In
the case that the dimension of the sphere is of the form $4n+3$,
then the sphere admits more than one vector field. The sphere
$S^{2n+1}$ possess the vector field
$$V_{n+1}(y)=-y_1\partial_{y_0}+y_0\partial_{y_1}-y_3\partial_{y_2}+\ldots-y_{2n+2}\partial_{y_{2n+1}}+y_{2n+1}\partial_{y_{2n+2}}.$$
Observe that this vector field has appeared already in two
opportunities: the vector field $V$ in Sections \ref{S3sR},
\ref{S3CR} and \ref{S3pb} corresponds to $V_2$; and the vector
field $Y_1$ in Subsection~\ref{rivf} of the Appendix corresponds
to $V_4$.

The vector field $V_{n+1}$ has the remarkable property that it
encloses valuable information concerning the CR structure of
$S^{2n+1}$. We know by Lemma \ref{lemmaCR} that, as a smooth
hypersurface in $\comp^{n+1}$ the sphere $S^{2n+1}$ admits a
holomorphic tangent space of dimension $$\dim_\real
H_pS^{2n+1}=2n$$ for any point $p\in S^{2n+1}$. The following
lemma implies the description of $H_pS^{2n+1}$ as the orthogonal
complement to $V_{n+1}$.

\begin{lemma}\label{alglin}
Let $W$ be an Euclidean space of dimension $n+2$, $n\geq1$, and
inner product $\langle\cdot,\cdot\rangle_W$. Consider an
orthogonal decomposition $W=\spn\{X,Y\}\oplus_\bot\tilde W$ with
respect to $\langle\cdot,\cdot\rangle_W$ and an orthogonal
endomorphism $A:W\to W$ such that $$A(\spn\{X,Y\})=\spn\{X,Y\},$$
then $\tilde W$ is an invariant space under the action of $A$,
i.e. $$A(\tilde W)=\tilde W.$$
\end{lemma}

\begin{proof}
Let $v\in\tilde W$, then for any $\alpha,\beta\in\real$ it is
clear that $$\langle Av,\alpha X+\beta Y\rangle_W=\langle
v,A^t(\alpha X+\beta Y)\rangle_W=\langle v,A^{-1}(\alpha X+\beta
Y)\rangle_W.$$

Since $A(\spn\{X,Y\})=\spn\{X,Y\}$, there exist real numbers $a,b$
such that $A^{-1}(\alpha X+\beta Y)=aX+bY$, and therefore
$$\langle Av,\alpha X+\beta Y\rangle_W=\langle
v,aX+bY\rangle_W=0$$ which implies that $Av\in\tilde W$.
\end{proof}

As an application of Lemma~\ref{alglin}, it is possible to obtain
an explicit characterization of $H_pS^{2n+1}$.

\begin{lemma}\label{hsphere}
For any $p\in S^{2n+1}$, the vector space $H_pS^{2n+1}$ is the
orthogonal complement to the vector $V_{n+1}(p)$ in $T_pS^{2n+1}$.
\end{lemma}

\begin{proof}
Consider the vector space $$W_p=\spn\{N_{n+1}(p)\}\oplus_{\bot}
T_pS^{2n+1}\cong T_p\real^{2n+2},$$ where $N_{n+1}(p)$ is the
normal vector to $S^{2n+1}$ at $p$. The standard almost complex
structure map $J_{n+1}:W_p\to W_p$ is clearly orthogonal. Moreover
$$J_{n+1}(V_{n+1}(p))=-N_{n+1}(p),\quad
J_{n+1}(N_{n+1}(p))=V_{n+1}(p).$$ Using the decomposition
$W_p=\tilde W_p\oplus_\bot\spn\{V_{n+1}(p),N_{n+1}(p)\}$ it is
possible to apply Lemma~\ref{alglin} in order to conclude that
$\tilde W_p$, which is the orthogonal complement to $V_{n+1}(p)$
in $T_pS^{2n+1}$, is invariant under $J_{n+1}$. Since
$\dim_\real\tilde W_p=2n$, we conclude that $\tilde
W_p=H_pS^{2n+1}$.
\end{proof}

\paragraph{\bf{Remark:}} The space $HS^{2n+1}$ can also be described
as the kernel of one form $$\omega=\bar z_0dz_0+\ldots+ \bar
z_ndz_n.$$ Indeed, take $X\in HS^{2n+1}$, then by straightforward
calculations we have
$$\omega(X)=\langle X,N_{n+1}\rangle+i\langle X,V_{n+1}\rangle=0.$$

\smallskip

Lemma~\ref{hsphere} provides a horizontal distribution of rank
$2n$ for the spheres $S^{2n+1}$, by considering the holomorphic
tangent space. The question now is whether this distribution is
bracket generating. The bracket generating condition for $S^3$ is
already discussed in Section~\ref{S3sR}. Here we present a basis
for bracket generating distribution of rank 6 for $S^7$. The
following ideas were introduced to the authors by Prof. K.
Furutani in a private communication.

\begin{theo}\label{rk6st2}
The subbundle $\dist=\spn\{Y_2,\ldots,Y_7\}$ of $TS^7$ is a bracket generating distribution of rank 6 and step 2.
\end{theo}

\begin{proof}
Define the following vector fields
\begin{eqnarray*}
v_{41}(y)&=&-y_4\partial_{y_0}+y_5\partial_{y_1}+y_0\partial_{y_4}-y_1\partial_{y_5},\\
v_{42}(y)&=& y_6\partial_{y_2}-y_7\partial_{y_3}-y_2\partial_{y_6}+y_3\partial_{y_7},\\
v_{51}(y)&=&-y_5\partial_{y_0}-y_4\partial_{y_1}+y_1\partial_{y_4}+y_0\partial_{y_5},\\
v_{52}(y)&=&-y_7\partial_{y_2}-y_6\partial_{y_3}+y_3\partial_{y_6}+y_0\partial_{y_7},
\end{eqnarray*}
and observe that $v_{41}+v_{42}=Y_4$ and $v_{51}+v_{52}=Y_5$. By straightforward calculations we see that
$$\langle v_{41}(y),Y_0(y)\rangle_y=\langle v_{42}(y),Y_0(y)\rangle_y=\langle v_{51}(y),Y_0(y)\rangle_y=\langle v_{52}(y),Y_0(y) \rangle_y=0$$ $$\langle v_{41}(y),Y_1(y)\rangle_y=\langle v_{42}(y),Y_1(y)\rangle_y=\langle v_{51}(y),Y_1(y)\rangle_y=\langle v_{52}(y),Y_1(y)\rangle_y=0$$ which implies that $v_{41},v_{42},v_{51},v_{52}\in \spn\{Y_2,\ldots,Y_7\}$. The following commutation relation $$[v_{41},v_{51}]+[v_{42},v_{52}]=-2Y_1$$ implies that the distribution $\dist$ is bracket generating of step 2.
\end{proof}

\paragraph{\bf{Remark:}} It is possible to repeat the previous argument with other pairs of vector fields. For example, if instead of $Y_4$ and $Y_5$ we employ $Y_2$ and $Y_3$, we can consider the vector fields
\begin{eqnarray*}
v_{21}(y)&=&-y_4\partial_{y_0}+y_5\partial_{y_1}+y_0\partial_{y_4}-y_1\partial_{y_5},\\
v_{22}(y)&=& y_6\partial_{y_2}-y_7\partial_{y_3}-y_2\partial_{y_6}+y_3\partial_{y_7},\\
v_{31}(y)&=&-y_5\partial_{y_0}-y_4\partial_{y_1}+y_1\partial_{y_4}+y_0\partial_{y_5},\\
v_{32}(y)&=&-y_7\partial_{y_2}-y_6\partial_{y_3}+y_3\partial_{y_6}+y_0\partial_{y_7}.
\end{eqnarray*}
We can proceed in a similar way if we use $Y_6$ and $Y_7$.

\s

We conclude this section by proving that the line bundle
$\spn\{V_{n+1}\}$ is the vertical space for the submersion given
by the Hopf fibration $S^1\hookrightarrow S^7 \buildrel
{h}\over\longrightarrow \comp P^n$. This implies that the
distribution $\dist$ defined in Theorem \ref{rk6st2} is an
Ehresmann connection for $h$. To achieve this, we recall that the
charts defining the holomorphic structure of $\comp P^n$ are given
by the open sets
$$U_k=\{[z_0:\ldots:z_n]:z_k\neq0\},$$ together with the
homeomorphisms
$$\begin{array}{ccccc}\varphi_k&:&U_k&\to&\comp^n\\
&&[z_0:\ldots:z_n]&\mapsto&(\frac{z_0}{z_k},\ldots,\frac{z_{k-1}}{z_k},\frac{z_{k+1}}{z_k},\ldots,\frac{z_n}{z_k}).\end{array}$$
Then, without loss of generality we will assume that $n=3$ and we
will develop the explicit calculations for $k=0$. The other cases
can be treated similarly.

Using the chart $(U_0,\varphi_0)$ defined above, we have the map
$$\begin{array}{ccccc}\varphi_0\circ h&\colon &S^{7}&\to&\comp^3\\
&&(z_0,z_1,z_2,z_3)&\mapsto&(\frac{z_1}{z_0},\frac{z_2}{z_0},\frac{z_3}{z_0}),\end{array}$$
which in real coordinates can be written as
$$\mbox{{$\varphi_0\circ
h(x_0,\ldots,x_7)=\left(\frac{x_0x_2+x_1x_3}{x_0^2+x_1^2},\frac{x_0x_3-x_1x_2}{x_0^2+x_1^2},
\frac{x_0x_4+x_1x_5}{x_0^2+x_1^2},\frac{x_0x_5-x_1x_4}{x_0^2+x_1^2},\frac{x_0x_6+x_1x_7}{x_0^2+x_1^2},\frac{x_0x_7-x_1x_6}{x_0^2+x_1^2}\right).$}}$$
The differential of this mapping is given by the matrix
$$d(\varphi_0\circ h)=$$ $$\mbox{\footnotesize{$\left(\begin{array}{cccccccc}
\frac{(x_1^2-x_0^2)x_2-2x_0x_1x_3}{(x_0^2+x_1^2)^2}&\frac{(x_0^2-x_1^2)x_3-2x_0x_1x_2}{(x_0^2+x_1^2)^2}&\frac{x_0}{x_0^2+x_1^2}&\frac{x_1}{x_0^2+x_1^2}&0&0&0&0\\
\frac{(x_1^2-x_0^2)x_3+2x_0x_1x_2}{(x_0^2+x_1^2)^2}&\frac{(x_1^2-x_0^2)x_2-2x_0x_1x_3}{(x_0^2+x_1^2)^2}&-\frac{x_1}{x_0^2+x_1^2}&\frac{x_0}{x_0^2+x_1^2}&0&0&0&0\\
\frac{(x_1^2-x_0^2)x_4-2x_0x_1x_5}{(x_0^2+x_1^2)^2}&\frac{(x_0^2-x_1^2)x_5-2x_0x_1x_4}{(x_0^2+x_1^2)^2}&0&0&\frac{x_0}{x_0^2+x_1^2}&\frac{x_1}{x_0^2+x_1^2}&0&0\\
\frac{(x_1^2-x_0^2)x_5+2x_0x_1x_4}{(x_0^2+x_1^2)^2}&\frac{(x_1^2-x_0^2)x_4-2x_0x_1x_5}{(x_0^2+x_1^2)^2}&0&0&-\frac{x_1}{x_0^2+x_1^2}&\frac{x_0}{x_0^2+x_1^2}&0&0\\
\frac{(x_1^2-x_0^2)x_6-2x_0x_1x_7}{(x_0^2+x_1^2)^2}&\frac{(x_0^2-x_1^2)x_7-2x_0x_1x_6}{(x_0^2+x_1^2)^2}&0&0&0&0&\frac{x_0}{x_0^2+x_1^2}&\frac{x_1}{x_0^2+x_1^2}\\
\frac{(x_1^2-x_0^2)x_7+2x_0x_1x_6}{(x_0^2+x_1^2)^2}&\frac{(x_1^2-x_0^2)x_6-2x_0x_1x_7}{(x_0^2+x_1^2)^2}&0&0&0&0&-\frac{x_1}{x_0^2+x_1^2}&\frac{x_0}{x_0^2+x_1^2}\end{array}\right).$}}$$

By straightforward calculations, we know that
$$\det([d(\varphi_0\circ H)][d(\varphi_0\circ
H)]^t)=(x_0^2+x_1^2)^{-8},$$ therefore, the matrix
$d(\varphi_0\circ H)$ has rank 6 and the dimension of its kernel
is 2: $$\dim_\real\ker d(\varphi_0\circ H)=2.$$ Moreover, since
$$d(\varphi_0\circ H)(N_{n+1})=d(\varphi_0\circ H)(V_{n+1})=0,$$
by direct calculations, we conclude $$\ker d(\varphi_0\circ
H)=\spn\{N_{n+1},V_{n+1}\}.$$ This implies that
$$\ker dH=\spn\{V_{n+1}\}.$$

\section{Application of the first quaternionic Hopf fibration}

Trying to imitate the work already done for $S^3$, we
find through the quaternionic Hopf bundle $S^3\to S^7\to S^4$ one of the
natural choice of horizontal distributions. We consider the quaternionic Hopf map given by

\begin{equation}\label{qHopf}\begin{array}{ccccc}h&:&S^7&\to&S^4\\&&(z,w)&\mapsto&(|z|^2-|w|^2,2z\bar{w})
\end{array}.\end{equation}

It can be written in real coordinates: \begin{eqnarray}\label{rqHopf} h(x_0,\ldots,x_7)  =
(x_0^2+x_1^2+x_2^2+x_3^2-x_4^2-x_5^2-x_6^2-x^7, \\
2 (x_0x_4+x_1x_5+x_2x_6+x_3x_7), 2 (-x_0x_5+x_1x_4-x_2x_7+x_3x_6),\nonumber
\\ 2(-x_0x_6+x_1x_7+x_2x_4-x_3x_5),  2(-x_0x_7-x_1x_6+x_2x_5+x_3x_4))\nonumber .\end{eqnarray}
The differential map $dh$ of $h$ is the following:

$$dh=2\left(\begin{array}{cccccccc}x_0&x_1&x_2&x_3&-x_4&-x_5&-x_6&-x_7\\
x_4&x_5&x_6&x_7&x_0&x_1&x_2&x_3\\-x_5&x_4&-x_7&x_6&x_1&-x_0&x_3&-x_2\\
-x_6&x_7&x_4&-x_5&x_2&-x_3&-x_0&x_1\\-x_7&-x_6&x_5&x_4&x_3&x_2&-x_1&-x_0
\end{array}\right).$$

Since no one of the commutators $[Y_i,Y_j]$, $i,j=1,\ldots,7$ coincides with the $Y_k$, $k=1,\ldots,7$, we look for the kernel of $dh$ among the commutators $Y_{ij}$, $i,j=1,\ldots,7$. We found that
$[dh]Y_{45}=[dh]Y_{46}=[dh]Y_{56}=0$. Define $V=\{Y_{45},Y_{46},Y_{56}\}$.

Our next step is to find the horizontal distribution $\spn\{\mathcal H\}$ that is transverse to $\spn\{V\}$ and bracket generating: $\spn\{\mathcal H\}_p\oplus\spn\{V\}_p=T_pS^7$ for all $p\in S^7$. To begin we define five basis for horizontal distributions, that we shall work with. The numeration is valid only for this section.

$$\mathcal H_0=\{Y_{47},Y_{57},Y_{67}, W\},$$ $$\mathcal H_1=\{Y_{34},Y_{35},Y_{36}, Y_{37}\},\quad \mathcal H_2=\{Y_{24},Y_{25},Y_{26}, Y_{27}\},$$ $$ \mathcal H_3=\{Y_{14},Y_{15},Y_{16}, Y_{17}\},\quad \mathcal H_4=\{Y_{04},Y_{05},Y_{06}, Y_{07}\},$$ where the vector field $W$ will be defined later and the notation $Y_{0k}=Y_{k}$ is chosen for the convenience.

We collect some useful information about sets $\mathcal H_m$, $m=0,\ldots,4$, that we exploit later.
\begin{itemize}
\item[1.]{All vector fields inside $\mathcal H_m$, $m=0,1,2,3,4$ are orthonormal (we do not count $W$ before we precise it).}
\item[2.]{All of collections $\mathcal H_m$, $m=0,1,2,3,4$ are bracket generating with the following commutator relations: $$\frac{1}{2}[Y_{j4},Y_{j5}]=Y_{45},\quad \frac{1}{2}[Y_{j4},Y_{j6}]=Y_{46},\quad \frac{1}{2}[Y_{j5},Y_{j6}]=Y_{56},\quad j=0,1,2,3, $$ $$\frac{1}{2}[Y_{47},Y_{57}]=Y_{45},\quad \frac{1}{2}[Y_{47},Y_{67}]=Y_{46},\quad \frac{1}{2}[Y_{57},Y_{67}]=Y_{56}.$$}
\item[3.]{We aim to calculate the angles between the vector fields from $\mathcal H_m$, $m=0,1,2,3,4$ and between vector fields from $\mathcal H_m$ and $V$. Beforehand, we introduce the following notations for the coordinates on the sphere $S^4$ given by the Hopf map $S^3\to S^7\to S^4$.
\begin{eqnarray}\label{eq:0}
a_{00} & = & y_0^2+y_1^2+y_2^2+y_3^2-y_4^2-y_5^2-y_6^2-y^2_7,  \nonumber \\
a_{11} & = & 2(y_0y_4+y_1y_5+y_2y_6+y_3y_7), \nonumber \\
a_{22} & = & 2(-y_0y_5+y_1y_4-y_2y_7+y_3y_6),\\
a_{33} & = & 2(-y_0y_6+y_1y_7+y_2y_4-y_3y_5),  \nonumber \\
a_{44} & = & 2(-y_0y_7-y_1y_6+y_2y_5+y_3y_4).\nonumber
\end{eqnarray} The first index of $a_{mk}$ reflects the number of the collection $\mathcal H_m$, where they will appear and the second one is related to the number of the coordinate on $S^4$.
}
\end{itemize}
We start from $\mathcal H_0$ and calculate the inner products:
\begin{equation*}\label{eq:00}
 \langle Y_{45},Y_{67} \rangle=-\langle Y_{46},Y_{57} \rangle=\langle Y_{56},Y_{47} \rangle=a_{00}.
\end{equation*} All other vector fields are orthogonal. We continue for $\mathcal H_1$.
\begin{equation}\label{eq:1}
\begin{array}{lllll}
 \langle Y_{45},Y_{36} \rangle & = & - \langle Y_{46},Y_{35} \rangle
\ \ = \ \ \langle Y_{56},Y_{34} \rangle & = & a_{11}  \\
 \langle Y_{45},Y_{37} \rangle & = & 2(-y_0y_5+y_1y_4+y_2y_7-y_3y_6) & = & a_{12} \\
 \langle Y_{46},Y_{37} \rangle & = & 2(-y_0y_6-y_1y_7+y_2y_4+y_3y_5) & = & a_{13} \\
 \langle Y_{56},Y_{37} \rangle & = & 2(y_0y_7-y_1y_6+y_2y_5-y_3y_4) & = & a_{14}. \\
\end{array}
\end{equation} All other vector fields in $\mathcal H_1\cup V$ are orthogonal. For the set $\mathcal H_2$ we see the following:
\begin{equation}\label{eq:2}
\begin{array}{lllll}
- \langle Y_{45},Y_{26} \rangle & = &  \langle Y_{46},Y_{25} \rangle
\ \ = \ \ - \langle Y_{56},Y_{24} \rangle & = & a_{22}  \\
 \langle Y_{45},Y_{27} \rangle & = & 2(y_0y_4+y_1y_5-y_2y_6-y_3y_7) & = & a_{21} \\
 \langle Y_{46},Y_{27} \rangle & = & 2(-y_0y_7+y_1y_6+y_2y_5-y_3y_4) & = & a_{24} \\
 \langle Y_{56},Y_{27} \rangle & = & 2(-y_0y_6-y_1y_7-y_2y_4-y_3y_5) & = & a_{23} \\
\end{array}
\end{equation} The other products between vector fields from $\mathcal H_2\cup V$ vanish. For $\mathcal H_3$ the situation is very similar.
\begin{equation}\label{eq:3}
\begin{array}{lllll}
\langle Y_{45},Y_{16} \rangle & = & - \langle Y_{46},Y_{15} \rangle
\ \ = \ \  \langle Y_{56},Y_{14} \rangle & = & a_{33}  \\
 \langle Y_{45},Y_{17} \rangle & = & 2(-y_0y_7-y_1y_6-y_2y_5-y_3y_4) & = & a_{34} \\
 \langle Y_{46},Y_{17} \rangle & = & 2(-y_0y_4+y_1y_5-y_2y_6+y_3y_7) & = & a_{31} \\
 \langle Y_{56},Y_{17} \rangle & = & 2(-y_0y_5-y_1y_4+y_2y_7+y_3y_6) & = & a_{32} \\
\end{array}
\end{equation}
All other vector fields from $\mathcal H_3\cup V$ are orthogonal. For the last collection $\mathcal H_4$ we obtain.
\begin{equation}\label{eq:4}
\begin{array}{lllll}
\langle Y_{45},Y_{06} \rangle & = & - \langle Y_{46},Y_{05} \rangle
\ \ = \ \  \langle Y_{56},Y_{04} \rangle & = & a_{44}  \\
 \langle Y_{45},Y_{07} \rangle & = & 2(y_0y_6-y_1y_7+y_2y_4-y_3y_5) & = & a_{43} \\
 \langle Y_{46},Y_{07} \rangle & = & 2(-y_0y_5-y_1y_4-y_2y_7-y_3y_6) & = & a_{42} \\
 \langle Y_{56},Y_{07} \rangle & = & 2(y_0y_4-y_1y_5-y_2y_6+y_3y_7) & = & a_{41} \\
\end{array}
\end{equation} with the rest of the product vanishing.

We notice some relations between the coefficients $a_{mk}$. The coordinates on $S^4$ possesses the equality
\begin{equation}\label{hopfcoord}
a_{00}^2+a_{11}^2+a_{22}^2+a_{33}^2+a_{44}^2=1.
\end{equation} The direct calculations also show
\begin{eqnarray}\label{cos}
a_{00}^2+a_{11}^2+a_{12}^2+a_{13}^2+a_{14}^2 & = & 1\nonumber \\
a_{00}^2+a_{21}^2+a_{22}^2+a_{23}^2+a_{24}^2 & = & 1 \\
a_{00}^2+a_{31}^2+a_{32}^2+a_{33}^2+a_{34}^2 & = & 1\nonumber \\
a_{00}^2+a_{41}^2+a_{42}^2+a_{43}^2+a_{44}^2 & = & 1.\nonumber
\end{eqnarray} In other words the sum of the squares of the cosines between vector fields from $\mathcal H_m\cup V$, $m=1,2,3,4$ is equal to $1-a_{00}^{2}$. Let us consider 2 cases: $0<a_{00}^{2}\leq 1$ and $a_{00}^{2}=0$.

\medskip

{\sc Case $0<a_{00}^{2}\leq 1$.}

\medskip

This case corresponds to any point on $S^4$
except of the set \begin{equation}\label{S1}S_1=\{y_0^2+y_1^2+y_2^2+y_3^2=y_4^2+y_5^2+y_6^2+y^2_7=1/2\}.\end{equation} We observe that the sum of the square of the cosines from~\eqref{cos}: $$\sum_{k=1}^{4}a_{mk}^2=1-a_{00}^2, \quad m=1,2,3,4$$ belongs to the interval $(0,1)$ and no one of the cosines can be equal to $1$. We conclude that each of $\mathcal H_m$, $m=1,2,3,4$, is transverse to $V$. Particularly, if $a_{00}^{2}=1$ then $\sum_{k=1}^{4}a_{mk}^2=0$ and $\mathcal H_m\bot V$. The latter situation occurs in the antipodal points $(\pm 1,0,0,0,0)\in S^4$ or is to say on the set \begin{equation}\label{S2}S_2=\{y_0^2+y_1^2+y_2^2+y_3^2=0,\ y_4^2+y_5^2+y_6^2+y^2_7=1\}
\ \cup$$ $$ \{y_0^2+y_1^2+y_2^2+y_3^2=1,\ y_4^2+y_5^2+y_6^2+y^2_7=0\}\in S^7.\end{equation}

We also can consider a collection $\mathcal H_0$, as a possible horizontal bracket generating distribution, if we choose an adequate vector field~$W$. If $a_{00}\in(0,1)$ we have $$0<a_{41}^2+a_{42}^2+a_{43}^2+a_{44}^2=1-a_{00}^2<1$$ and no one of the products in~\eqref{eq:4} can give $1$. We conclude that $Y_{07}$ can not be collinear to $V=\{Y_{45},Y_{46},Y_{56}\}$. Therefore, we choose $W=Y_{07}$. By the same reason we could take $Y_{j7}$, $j=1,2,3$. In the case when $a_{00}^{2}=1$ the vector fields $Y_{j7}$, $j=0,1,2,3$ are orthogonal to~$V$ but $\mathcal H_0$ is collinear to $V$ and the collection $\mathcal H_0$ with $W=Y_{j7}$ is not transverse to $V$.

\medskip

{\sc Case $a_{00}^{2}=0$.}

In this case the distribution $\mathcal H_0$ is nicely serve as a bracket generating if we find a suitable vector field $W$. Notice that~\eqref{hopfcoord} became
\begin{equation}\label{hopfcoord0}
a_{11}^2+a_{22}^2+a_{33}^2+a_{44}^2=1.
\end{equation} The $a_{mm}$ can not vanish simultaneously. Without lost of generality, we can assume that $a_{44}\neq 0$. Then $a^2_{41}+a^2_{42}+a^2_{43}=1-a^2_{44}<1$ from~\eqref{cos} and the products~\eqref{eq:4} imply that $Y_{07}$ is transverse to $V$ and can be used as a vector field $W$. In the case $a_{44}^2=1$ we get that $Y_{07}$ is orthogonal to $V$. Since $W\bot Y_{j7}$, $j=0,\ldots,3$ the collection $\mathcal H_0$ with any choice of $Y_{j7}$, $j=0,\ldots,3$ will be orthonormal.

We formulate the latter result in the following theorem

\begin{theo} Let~\eqref{qHopf} be the Hopf map with the vertical space $$V=\{Y_{45},Y_{46},Y_{56}\},$$ $S_1$ and $S_2$ are given by~\eqref{S1} and~\eqref{S2}. Then the Hoph map produces the following Ehresmann Connection $\mathcal H_p$, $p\in S^7$:
\begin{itemize}
\item [(i)]{if $p\notin S_1$ then $\mathcal H_p=(\mathcal H_m)_p$, for any choice of $m=1,2,3,4$;}
\item[(ii)]{if $p\notin S_2$ then $\mathcal H_p=(\mathcal H_0\cup Y_{j7})_p$, $j=0,1,2,3$;}
\end{itemize} and we have $$\spn\{\mathcal H_0,Y_{j7}\}_p\oplus\spn\{V\}_p=T_pS^7,\ j=0,1,2,3 \quad\text{for all}\quad p\in S^7.$$
\end{theo}

\section{Appendix}

\subsection{Multiplication of octonions}\label{multoct}

Let $$o_1=(x_0e_0+x_1e_1+x_2e_2+x_3e_3+x_4e_4+x_5e_5+x_6e_6+x_7e_7)$$
and $$o_2=(y_0e_0+y_1e_1+y_2e_2+y_3e_3+y_4e_4+y_5e_5+y_6e_6+y_7e_7)$$ be two
octonions. Then we have according to Table~\ref{oct}

$$o_1\circ o_2=(x_0e_0+x_1e_1+x_2e_2+x_3e_3+x_4e_4+x_5e_5+x_6e_6+x_7e_7)\circ$$
$$(y_0e_0+y_1e_1+y_2e_2+y_3e_3+y_4e_4+y_5e_5+y_6e_6+y_7e_7)=$$
$$=(x_0y_0-x_1y_1-x_2y_2-x_3y_3-x_4y_4-x_5y_5-x_6y_6-x_7y_7)e_0+$$
$$+(x_1y_0+x_0y_1-x_3y_2+x_2y_3-x_5y_4+x_4y_5+x_7y_6-x_6y_7)e_1+$$
$$+(x_2y_0+x_3y_1+x_0y_2-x_1y_3-x_6y_4-x_7y_5+x_4y_6+x_5y_7)e_2+$$
$$+(x_3y_0-x_2y_1+x_1y_2+x_0y_3-x_7y_4+x_6y_5-x_5y_6+x_4y_7)e_3+$$
$$+(x_4y_0+x_5y_1+x_6y_2+x_7y_3+x_0y_4-x_1y_5-x_2y_6-x_3y_7)e_4+$$
$$+(x_5y_0-x_4y_1+x_7y_2-x_6y_3+x_1y_4+x_0y_5+x_3y_6-x_2y_7)e_5+$$
$$+(x_6y_0-x_7y_1-x_4y_2+x_5y_3+x_2y_4-x_3y_5+x_0y_6+x_1y_7)e_6+$$
$$+(x_7y_0+x_6y_1-x_5y_2-x_4y_3+x_3y_4+x_2y_5-x_1y_6+x_0y_7)e_7.$$

\subsection{Right invariant vector fields}\label{rivf}

According to the previous multiplication rule, we have the
following unit vector fields of $\real^8$ arising as right
invariants vector fields under the octonion multiplication.
$$
Y_0(y)=y_0\partial_{y_0}+y_1\partial_{y_1}+y_2\partial_{y_2}+
y_3\partial_{y_3}+y_4\partial_{y_4}+y_5\partial_{y_5}+
y_6\partial_{y_6}+y_7\partial_{y_7}
$$ $$
Y_1(y)=-y_1\partial_{y_0}+y_0\partial_{y_1}-y_3\partial_{y_2}+
y_2\partial_{y_3}-y_5\partial_{y_4}+y_4\partial_{y_5}-
y_7\partial_{y_6}+y_6\partial_{y_7}
$$ $$
Y_2(y)=-y_2\partial_{y_0}+y_3\partial_{y_1}+y_0\partial_{y_2}-
y_1\partial_{y_3}-y_6\partial_{y_4}+y_7\partial_{y_5}+
y_4\partial_{y_6}-y_5\partial_{y_7}
$$ $$
Y_3(y)=-y_3\partial_{y_0}-y_2\partial_{y_1}+y_1\partial_{y_2}+
y_0\partial_{y_3}+y_7\partial_{y_4}+y_6\partial_{y_5}-
y_5\partial_{y_6}-y_4\partial_{y_7}
$$ $$
Y_4(y)=-y_4\partial_{y_0}+y_5\partial_{y_1}+y_6\partial_{y_2}-
y_7\partial_{y_3}+y_0\partial_{y_4}-y_1\partial_{y_5}-
y_2\partial_{y_6}+y_3\partial_{y_7}
$$ $$
Y_5(y)=-y_5\partial_{y_0}-y_4\partial_{y_1}-y_7\partial_{y_2}-
y_6\partial_{y_3}+y_1\partial_{y_4}+y_0\partial_{y_5}+
y_3\partial_{y_6}+y_2\partial_{y_7}
$$ $$
Y_6(y)=-y_6\partial_{y_0}+y_7\partial_{y_1}-y_4\partial_{y_2}+
y_5\partial_{y_3}+y_2\partial_{y_4}-y_3\partial_{y_5}+
y_0\partial_{y_6}-y_1\partial_{y_7}
$$ $$
Y_7(y)=-y_7\partial_{y_0}-y_6\partial_{y_1}+y_5\partial_{y_2}+
y_4\partial_{y_3}-y_3\partial_{y_4}-y_2\partial_{y_5}+
y_1\partial_{y_6}+y_0\partial_{y_7}.
$$

The vector fields $Y_i$, $i=1,\ldots,7$ form an orthonormal frame of $T_pS^7$, $p\in S^7$, with respect to restriction of the inner product $\langle\cdot,\cdot\rangle$ from $\mathbb R^8$ to the tangent space $T_pS^7$ at each $p\in S^7$.

\subsection{Commutators between right invariant vector fields}\label{crivf}
Denoting by $Y_{ij}(y)=\frac12[Y_i(y),Y_j(y)]$ the commutator between
the right invariant vector fields, described in the previous
subsection, we have the following list:
$$
Y_{12}(y)=y_3\partial_{y_0}+y_2\partial_{y_1}-y_1\partial_{y_2}-
y_0\partial_{y_3}+y_7\partial_{y_4}+y_6\partial_{y_5}-
y_5\partial_{y_6}-y_4\partial_{y_7}
$$ $$
Y_{13}(y)=-y_2\partial_{y_0}+y_3\partial_{y_1}+y_0\partial_{y_2}-
y_1\partial_{y_3}+y_6\partial_{y_4}-y_7\partial_{y_5}-
y_4\partial_{y_6}+y_5\partial_{y_7}
$$ $$
Y_{14}(y)=y_5\partial_{y_0}+y_4\partial_{y_1}-y_7\partial_{y_2}-
y_6\partial_{y_3}-y_1\partial_{y_4}-y_0\partial_{y_5}+
y_3\partial_{y_6}+y_2\partial_{y_7}
$$ $$
Y_{15}(y)=-y_4\partial_{y_0}+y_5\partial_{y_1}-y_6\partial_{y_2}+
y_7\partial_{y_3}+y_0\partial_{y_4}-y_1\partial_{y_5}+
y_2\partial_{y_6}-y_3\partial_{y_7}
$$ $$
Y_{16}(y)=y_7\partial_{y_0}+y_6\partial_{y_1}+y_5\partial_{y_2}+
y_4\partial_{y_3}-y_3\partial_{y_4}-y_2\partial_{y_5}-
y_1\partial_{y_6}-y_0\partial_{y_7}
$$ $$
Y_{17}(y)=-y_6\partial_{y_0}+y_7\partial_{y_1}+y_4\partial_{y_2}-
y_5\partial_{y_3}-y_2\partial_{y_4}+y_3\partial_{y_5}+
y_0\partial_{y_6}-y_1\partial_{y_7}
$$ $$
Y_{23}(y)=y_1\partial_{y_0}-y_0\partial_{y_1}+y_3\partial_{y_2}-
y_2\partial_{y_3}-y_5\partial_{y_4}+y_4\partial_{y_5}-
y_7\partial_{y_6}+y_6\partial_{y_7}
$$ $$
Y_{24}(y)=y_6\partial_{y_0}+y_7\partial_{y_1}+y_4\partial_{y_2}+
y_5\partial_{y_3}-y_2\partial_{y_4}-y_3\partial_{y_5}-
y_0\partial_{y_6}-y_1\partial_{y_7}
$$ $$
Y_{25}(y)=-y_7\partial_{y_0}+y_6\partial_{y_1}+y_5\partial_{y_2}-
y_4\partial_{y_3}+y_3\partial_{y_4}-y_2\partial_{y_5}-
y_1\partial_{y_6}+y_0\partial_{y_7}
$$ $$
Y_{26}(y)=-y_4\partial_{y_0}-y_5\partial_{y_1}+y_6\partial_{y_2}+
y_7\partial_{y_3}+y_0\partial_{y_4}+y_1\partial_{y_5}-
y_2\partial_{y_6}-y_3\partial_{y_7}
$$ $$
Y_{27}(y)=y_5\partial_{y_0}-y_4\partial_{y_1}+y_7\partial_{y_2}-
y_6\partial_{y_3}+y_1\partial_{y_4}-y_0\partial_{y_5}+
y_3\partial_{y_6}-y_2\partial_{y_7}
$$ $$
Y_{34}(y)=-y_7\partial_{y_0}+y_6\partial_{y_1}-y_5\partial_{y_2}+
y_4\partial_{y_3}-y_3\partial_{y_4}+y_2\partial_{y_5}-
y_1\partial_{y_6}+y_0\partial_{y_7}
$$ $$
Y_{35}(y)=-y_6\partial_{y_0}-y_7\partial_{y_1}+y_4\partial_{y_2}+
y_5\partial_{y_3}-y_2\partial_{y_4}-y_3\partial_{y_5}+
y_0\partial_{y_6}+y_1\partial_{y_7}
$$ $$
Y_{36}(y)=y_5\partial_{y_0}-y_4\partial_{y_1}-y_7\partial_{y_2}+
y_6\partial_{y_3}+y_1\partial_{y_4}-y_0\partial_{y_5}-
y_3\partial_{y_6}+y_2\partial_{y_7}
$$ $$
Y_{37}(y)=y_4\partial_{y_0}+y_5\partial_{y_1}+y_6\partial_{y_2}+
y_7\partial_{y_3}-y_0\partial_{y_4}-y_1\partial_{y_5}-
y_2\partial_{y_6}-y_3\partial_{y_7}
$$ $$
Y_{45}(y)=y_1\partial_{y_0}-y_0\partial_{y_1}-y_3\partial_{y_2}+
y_2\partial_{y_3}+y_5\partial_{y_4}-y_4\partial_{y_5}-
y_7\partial_{y_6}+y_6\partial_{y_7}
$$ $$
Y_{46}(y)=y_2\partial_{y_0}+y_3\partial_{y_1}-y_0\partial_{y_2}-
y_1\partial_{y_3}+y_6\partial_{y_4}+y_7\partial_{y_5}-
y_4\partial_{y_6}-y_5\partial_{y_7}
$$ $$
Y_{47}(y)=-y_3\partial_{y_0}+y_2\partial_{y_1}-y_1\partial_{y_2}+
y_0\partial_{y_3}+y_7\partial_{y_4}-y_6\partial_{y_5}+
y_5\partial_{y_6}-y_4\partial_{y_7}
$$ $$
Y_{56}(y)=-y_3\partial_{y_0}+y_2\partial_{y_1}-y_1\partial_{y_2}+
y_0\partial_{y_3}-y_7\partial_{y_4}+y_6\partial_{y_5}-
y_5\partial_{y_6}+y_4\partial_{y_7}
$$ $$
Y_{57}(y)=-y_2\partial_{y_0}-y_3\partial_{y_1}+y_0\partial_{y_2}+
y_1\partial_{y_3}+y_6\partial_{y_4}+y_7\partial_{y_5}-
y_4\partial_{y_6}-y_5\partial_{y_7}
$$ $$
Y_{67}(y)=y_1\partial_{y_0}-y_0\partial_{y_1}-y_3\partial_{y_2}+
y_2\partial_{y_3}-y_5\partial_{y_4}+y_4\partial_{y_5}+
y_7\partial_{y_6}-y_6\partial_{y_7}.
$$

\end{document}